\documentclass[twoside]{article} 
\usepackage{graphicx}
\usepackage{amsfonts}
\usepackage[fleqn]{amsmath}
\usepackage{amssymb}
\usepackage{amsthm}
\usepackage{amscd}
\usepackage[T1]{fontenc}
\usepackage{afterpage}  %
\usepackage{fancyhdr}
\usepackage{color}
\theoremstyle{plain}
\newtheorem{theorem}{Theorem}[section]
\newtheorem{lemma}[theorem]{Lemma}
\newtheorem{proposition}[theorem]{Proposition}

\newtheorem{corollary}[theorem]{Corollary}
\theoremstyle{definition}

\newtheorem{example}[theorem]{Example}
\theoremstyle{remark}

\begin{document}

\afterpage{\rhead[]{\thepage} \chead[\small W. A. Dudek and R. A. R. Monzo ]
{\small Pentagonal quasigroups, their translatability and parastrophes} \lhead[\thepage]{} }                  

\begin{center}
\vspace*{2pt}
{\Large \textbf{Pentagonal quasigroups, their translatability and parastrophes}}\\[30pt]
 {\large \textsf{\emph{Wieslaw A. Dudek \ and \ Robert A. R. Monzo}}}\\[30pt]
\end{center}
 {\footnotesize\textbf{Abstract.}  Any pentagonal quasigroup $Q$ is proved to have the product $xy=\varphi(x)+y-\varphi(y)$, where $(Q,+)$  is an Abelian group, $\varphi$ is its regular automorphism satisfying $\varphi^4-\varphi^3+\varphi^2-\varphi+\varepsilon=0$ and $\varepsilon$ is the identity mapping. All Abelian groups of order $n<100$ inducing pentagonal quasigroups are determined. The variety of commutative, idempotent, medial groupoids satisfying the pentagonal identity $(xy\cdot x)y\cdot x=y$ is proved to be the variety of commutative, pentagonal quasigroups, whose spectrum is $\{11^n:n=0,1,2,\ldots\}$. We prove that the only translatable commutative pentagonal quasigroup is $xy=(6x+6y)({\rm mod}\,11)$. The parastrophes of a pentagonal quasigroup are classified according to well-known types of idempotent translatable quasigroups. The translatability of a pentagonal quasigroup induced by the group $\mathbb{Z}_n$ and its automorphism $\varphi(x)=ax$ is proved to determine the value of $a$ and the range of values of $n$.}



\section{Introduction}
This paper was inspired by the work of Vidak in \cite{Vid}. It is also a continuation of the ideas appearing in \cite{DM2}. All results here follow from the main result, Theorem \ref{main}, which gives a new characterisation of a pentagonal quasigroup $(Q,\cdot)$ in terms of a regular automorphism $\varphi$ on an Abelian group $(Q,+)$, where $xy=\varphi(x)+y-\varphi(y)$, $\varphi^4-\varphi^3+\varphi^2-\varphi+\varepsilon=0$ and $\varepsilon$ is the identity mapping on $Q$. We say then that $(Q,+)$ {\em induces} the pentagonal quasigroup $(Q,\cdot)$. 

Notice that $xy=(\varepsilon-\varphi)(y)+x-(\varepsilon-\varphi)(x)$. The characterisation of a pentagonal quasigroup $(Q,\cdot)$ given by Vidak in \cite{Vid} is that $xy=\psi(y)+x-\psi(x)$ for some automorphism $\psi$ on an Abelian group $(Q,+)$, where $\psi^4-3\psi^3+4\psi^2-2\psi+\varepsilon=0$. Now since, when $\varphi^4-\varphi^3+\varphi^2-\varphi+\varepsilon=0$, $(\varepsilon-\varphi)^4-3(\varepsilon-\varphi)^3+4(\varepsilon-\varphi)^2-2(\varepsilon-\varphi)+\varepsilon=0$, we can think of $\psi$ as equal to $\varepsilon-\varphi$.

In Theorem \ref{T-Zn} we prove that a pentagonal quasigroup induced by the group $\mathbb{Z}_n$ has the form $xy=(ax+(1-a)y)({\rm mod}\,n)$, where $(a^4-a^3+a^2-a+1)=0({\rm mod}\,n)$. Vidak's identity gives the second component, namely $\psi(x)=(1-a)x({\rm mod}\,n)$.

As a consequence of our characterisation, all Abelian groups of order $n<100$ that induce pentagonal quasigroups are determined. Also, the variety of commutative, idempotent, medial groupoids satisfying the pentagonal identity $(xy\cdot x)y\cdot x=y$ is proved in Corollary \ref{Cc} to be the variety of commutative, pentagonal quasigroups, whose spectrum is $\{11^n:n=0,1,2,\ldots\}$. The form of commutative pentagonal quasigroups is determined in Proposition \ref{Zn-com} and as a corollary we prove that the only translatable commutative pentagonal quasigroup is $xy=(6x+6y)({\rm mod}\,11)$. In Theorem \ref{T-trans2} we prove that the translatability of a pentagonal quasigroup induced by the group $\mathbb{Z}_n$ and its automorphism $\varphi(x)=ax$ determines the value of $a$ and all the possible values of $n$.

Using results from \cite{DM3} in the last table we classify the parastrophes of pentagonal quasigroups in terms of well-known types of idempotent translatable quasigroups.


\section{Existence of pentagonal quasigroups}

All considered quasigroups are finite and have form $Q=\{1,2,\ldots,n\}$ with the natural ordering, which is always possible by renumeration of elements. For simplicity, instead of $(x+y)\equiv z({\rm mod}\,n)$ we write $[x+y]_n=[z]_n$. Also, in calculations modulo $n$ we identify $0$ with $n$.

According to \cite{Vid} a quasigroup $(Q,\cdot)$ is called {\em pentagonal} if it satisfies the following three identities:
\begin{eqnarray}
&&xx = x, \label{e1} \\
&&xy\cdot zu = xz\cdot yu, \label{e2} \\
&&(xy\cdot x)y\cdot x = y. \label{e3} 
\end{eqnarray}

Let's recall that a mapping $\varphi$ of a group $(Q,+)$ onto $(Q,+)$ is called {\it regular} if $\varphi(x)=x$ holds only for $x=0$.

Below we present a full characterization of pentagonal quasigroups.

\begin{theorem}\label{main}
A groupoid $(Q,\cdot)$ is a pentagonal quasigroup if and only if on $Q$ one can define an Abelian group $(Q,+)$ and its regular automorphism $\varphi$ such that
\begin{eqnarray}
&&x\cdot y=\varphi(x)+(\varepsilon-\varphi)(y), \label{e4}\\[3pt] 
&&\varphi^4-\varphi^3+\varphi^2-\varphi+\varepsilon=0, \label{e5} 
\end{eqnarray}
where $\varepsilon$ is the identity automorphism.
\end{theorem}
\begin{proof} By the Toyoda theorem (see for example \cite{Scerb}), any quasigroup $(Q,\cdot)$ satisfying \eqref{e1} and \eqref{e2} can be presented in the form \eqref{e4}, where $(Q,+)$ is an Abelian group and $\varphi$ is its automorphism. Applying this fact to \eqref{e3} and putting $y=0$ we obtain \eqref{e5}. From \eqref{e5} it follows that the automorphism $\varphi$ is regular.

Conversely, a groupoid $(Q,\cdot)$ defined by \eqref{e4}, where $\varphi$ is an automorphism of an Abelian group $(Q,+)$, is a quasigroup satisfying \eqref{e1} and \eqref{e2}. Applying \eqref{e5} to $z=x-y$ and using \eqref{e4}, after simple calculations, we obtain \eqref{e3}.	\end{proof}

This means that pentagonal quasigroups are isotopic to the group inducing them. Thus, pentagonal quasigroups are isotopic if and only if they are induced by isomorphic groups. 

\begin{example}\rm 
Let $(\mathbb{C},+)$ be the additive group of complex numbers. Then $\varphi(z)=ze^{\frac{\pi}{5}i}$ is a regular automorphism of $(\mathbb{C},+)$ satisfying \eqref{e5}. Thus, by Theorem \ref{main}, the set of complex numbers with multiplication defined by \eqref{e4} is an infinite pentagonal quasigroup.
\end{example}

As a consequence of the above theorem we obtain

\begin{corollary}\label{C-abel} On a pentagonal quasigroup $(Q,\cdot)$ one can define an Abelian group $(Q,+)$ and its regular automorphism $\varphi$ such that
\eqref{e4} holds and 
\begin{eqnarray*}
&&\varphi^5+\varepsilon=0 \ \  and \ \ \varphi\ne -\varepsilon \ \ or\\[3pt]
&&\varphi=-\varepsilon \ \ and \ \  exp(Q,+)=5,
\end{eqnarray*}
where $\varepsilon$ is the identity automorphism.
\end{corollary} 

\begin{corollary}\label{phi}
An Abelian group inducing a pentagonal quasigroup is the direct product of cyclic groups of order $5$ or has a regular automorphism of order $10$.
\end{corollary}

The converse statement is not true. The automorphism $\varphi(x)=[4x]_{25}$ of the group $\mathbb{Z}_{25}$ is regular and satisfies the above condition, but $\mathbb{Z}_{25}$ with the multiplication $x\cdot y=[4x+22y]_{25}$ is not a pentagonal quasigroup.

\medskip
The following lemma is obvious.

\begin{lemma} \label{L-1}
The direct product of pentagonal quasigroups is also a pentagonal quasigroup.
\end{lemma}

\begin{corollary}\label{C-0}
For every $t$ there is a pentagonal quasigroup of order $5^t$.
\end{corollary}
\begin{proof}
For $t=0$ it is trivial quasigroup. For $t=1$ it is induced by the additive group $\mathbb{Z}_5$ and has the form $x\cdot y=[4x+2y]_5$. For $t>1$ it is the direct product of $t$ copies of the last quasigroup.
\end{proof}

\begin{proposition}\label{cprod}
If finite Abelian groups $G_1$ and $G_2$ have relatively prime orders, then any pentagonal quasigroup induced by the group $G_1\times G_2$ is the direct product of pentagonal quasigroups induced by groups $G_1$ and $G_2$.
\end{proposition}
\begin{proof}
If $G_1$ and $G_2$ have relatively prime orders, then, accordind to Lemma 2.1 in \cite{Hil},
${\rm Aut}(G_1\times G_2)\cong{\rm Aut}(G_1)\times{\rm Aut}(G_2)$. So, each automorphism $\varphi$ of $G_1\times G_2$ can be treated as an automorphism of the form $\varphi=(\varphi_1,\varphi_2)$, where $\varphi_1,\varphi_2$ are automorphisms of $G_1$ and $G_2$, respectively. Obviously, $\varphi$ is regular if and only if $\varphi_1$ and $\varphi_2$ are regular. Moreover, $\varphi$ satisfies \eqref{e5} if and only if $\varphi_1$ and $\varphi_2$ satisfy \eqref{e5}. Thus, a pentagonal quasigroup induced by $G_1\times G_2$ is the direct product of pentagonal quasigroups induced by $G_1$ and $G_2$.
\end{proof}

To determine Abelian groups that induce pentagonal quasigroups we will need the following theorem proved in \cite{Hil}. 

\begin{theorem}\label{T-Hil}
The Abelian group $G=\mathbb{Z}_{p^{\alpha_1}}\times\mathbb{Z}_{p^{\alpha_2}}\times\cdots\mathbb{Z}_{p^{\alpha_m}}$ has
$$
|{\rm Aut}(G)|=\prod_{k=1}^m(p^{d_k}-p^{k-1})\prod_{j=1}^m(p^{\alpha_j})^{m-d_j}\prod_{i=1}^m(p^{\alpha_i-1})^{m+1-c_i} \;,
$$
where $d_k=\max\{l:\alpha_l=\alpha_k\}$ and $c_k=\min\{l:\alpha_l=\alpha_k\}$.
\end{theorem}


\section{Construction of pentagonal quasigroups}\setcounter{section}{3}\setcounter{theorem}{0}

We start with the characterization of pentagonal quasigroups induced by $\mathbb{Z}_n$.
\begin{theorem}\label{main2}
A groupoid $(Q,\cdot)$ of order $n>2$ is a pentagonal quasigroup induced by the group $\mathbb{Z}_n$ if and only if there exist $1<a<n$ such that $(a,n)=(a-1,n)=1$, $x\cdot y=[ax+(1-a)y]_n$ and
\begin{eqnarray}\label{e7}
[a^4-a^3+a^2-a+1]_n=0 .
\end{eqnarray}
\end{theorem}
\begin{proof}
Automorphisms of the group $\mathbb{Z}_n$ have the form $\varphi(x)=ax$, where $(a,n)=1$.
Since $\varepsilon-\varphi$ also is an automorphism, $(a-1,n)=1$. Moreover, the equation
$(a-1)x=0({\rm mod}\,n)$ has $d=(a-1,n)$ solutions (cf. \cite{Vin}). So, $(a-1,n)=1$ means that the automorphism $\varphi(x)=ax$ is regular.
Theorem \ref{main} completes the proof.
\end{proof}

\begin{theorem}\label{T-constr}
If a regular automorphism $\varphi$ of an Abelian group $(Q,+)$ satisfies $\eqref{e5}$, then $(Q,\ast)$, $(Q,\circ)$ and $(Q,\diamond)$ with the operations
$$x\ast y=\varphi^2(y-x)+y, \  \ \ \ \ x\circ y=\varphi^3(x-y)+y, \ \ \ \ \ x\diamond y=\varphi^4(y-x)+y
$$
are pentagonal quasigroups.
\end{theorem}
\begin{proof}
If $\varphi$ and $(Q,+)$ are as in the assumption, then, by Theorem \ref{main}, $(Q,\cdot)$ with the operation $x\cdot y=\varphi(x-y)+y$ is a pentagonal quasigroup. From Vidak's results presented in \cite{Vid} it follows that also $(Q,\ast)$, $(Q,\circ)$ and $(Q,\diamond)$, where $x\ast y=y\cdot (yx\cdot x)x$, $x\circ y=(yx\cdot y)x$ and $x\diamond y= (xy\cdot )y$, are pentagonal quasigroups. Applying \eqref{e4} and \eqref{e5} to these operations we obtain our thesis.
\end{proof}

\begin{theorem}\label{T-Zn} 
A pentagonal quasigroup induced by the group $\mathbb{Z}_n$ has one of the following forms 
\begin{eqnarray*}
&&x\cdot y=[ax+(1-a)y]_n,\\
&&x\cdot y=[-a^2x+(1+a^2)y]_n,\\
&&x\cdot y=[a^3x+(1-a^3)y]_n,\\
&&x\cdot y=[-a^4x+(1+a^4)y]_n,
\end{eqnarray*} 
where $1<a<n-1$ satisfy $\eqref{e7}$ and $(a,n)=(a-1,n)=1$.

When $a=n-1$ there is only one pentagonal quasigroup. It is induced by $\mathbb{Z}_5$ and has the form $x\cdot y=[4x+2y]_5$.
\end{theorem}
\begin{proof}
Equation \eqref{e7} has no more than four solutions, so $\mathbb{Z}_n$ induces no more than four pentagonal quasigroups. Theorems \ref{main2} and \ref{T-constr} complete the proof for $1<a<n-1$. 
The case $a=n-1$ is obvious.
\end{proof}

Note that for $[a+1]_n\ne 0$ the equation \eqref{e7} implies $[a^5+1]_n=0$. Since it is valid also for $a=-1$ in a pentagonal quasigroup with $x\cdot y=[ax+(1-a)y]_n$ we have
\begin{eqnarray}\label{e8}
[a^5]_n=n-1.
\end{eqnarray}
\begin{proposition}\label{sub}
Let $(Q,\cdot)$ be a pentagonal quasigroup induced by the group $\mathbb{Z}_n$, where $n>5$. If $m|n$, then $(Q,\cdot)$ contains a pentagonal subquasigroup of order $m$.
\end{proposition}
\begin{proof}
If $m|n$ then the group $\mathbb{Z}_n$ contains a subgroup isomorphic to $\mathbb{Z}_m$. Let $x\cdot y=[ax+(1-a)y]_n$.
Since $1<a<n-1$, $(a,n)=(a-1,n)=1$, also $(a,m)=(a-1,m)=1$ and $[a^4-a^3+a^2-a+1]_m=0$. Let $a'=[a]_m$. Then, as it is not difficult to see, $\mathbb{Z}_m$ with the multiplication $x\cdot y=[a'x+(1-a')y]_m$ is a pentagonal quasigroup.
\end{proof}

\begin{proposition}\label{ord}
If an Abelian group $G$ inducing a pentagonal quasigroup has an element of order $k>1$, then the number of such elements is greater than $3$.
\end{proposition}
\begin{proof}
An automorphism preserves the order of elements of $G$. So, if only one $x\in G$ has order $k>1$, then $\varphi(x)=x$, which contradicts to the assumption on $\varphi$. If only two elements $x\ne y$ have order $k>1$, then $\varphi(x)=y$ and $\varphi^2(x)=x$. Using \eqref{e5} we get $3x=2y$ and $3y=2x$. Therefore, $2x=3y=y+3x$, which implies that $x=-y$. But $k(2x)=2(kx)=0$. Also $2x\ne 0$ or else $x=-x=y$, a contradiction. Thus, $2x$ has order $k$. Then $2x=x$ or $2x=y$. The first case is impossible. In the second $2x=y=3y$ implies $2y=0$, so $y=-y=x$, a contradiction. Therefore, $G$ has at least three elements of order $k$.

If $G$ has three distinct elements $x,y,z$ of order $k>1$, then $\varphi(x)=y$, $\varphi^2(x)=\varphi(y)=z$, $\varphi^3(x)=\varphi(z)=x.$ Obviously, $\varphi(x)\ne -x$, because $\varphi(x)=-x$ implies $x=\varphi^2(x)=z$, which is impossible. Thus, by Corollary \ref{C-abel}, $0=\varphi^5(x)+x=z+x$ and $0=\varphi(z)+\varphi(x)=x+y$. So, $x+z=x+y$, a contradiction.
Hence, $G$ has more than three elements of order $k>1$.
\end{proof}
\begin{corollary}\label{C-ord}
Abelian groups of order $n$, where 
\begin{enumerate}
\item[$(i)$] $2|n$ and $4/\!\! |\,n$ or 
\item[$(ii)$] $3|n$ and $9/\!\! |\,n$ or 
\item[$(iii)$] $4|n$ and $8/\!\! |\,n$,
\end{enumerate}
do not induce pentagonal quasigroups.
\end{corollary}
\begin{proof}
In the  first case, a group has one element of order $2$; in the second -- two elements of order $3$; in third case -- one or three elements of order $2$.
\end{proof}
\begin{theorem}\label{T-fin}
A finite pentagonal quasigroup has order $5s$ or $5s+1$.
\end{theorem}
\begin{proof}
Suppose that a pentagonal quasigroup $(Q,\cdot)$ is induced by the group $(Q,+)$, where $Q=\{0,e_2,e_3,\ldots,e_n\}$. 
Each automorphism $\psi$ of this group can be identified with a permutation $\varphi$ of the set $\{e_2,e_3,\ldots,e_n\}$. Each such permutation is a cycle or can be decomposed into disjoint cycles. Since, by Corollary \ref{C-abel}, $\varphi^2=\varepsilon$ or $\varphi^{10}=\varepsilon$, a permutation $\varphi$ can be decomposed into disjoint cycles of the length $2$, $5$ or $10$. If $\varphi$ contains a cycle of the length $2$, then for some $e_i\in Q$ we have $\varphi(e_i)=e_j\ne e_i$ and $\varphi^2(e_i)=e_i$. If $e_j\ne -e_i$, then by Corollary \ref{C-abel}, $-e_i=\varphi^5(e_i)=\varphi(e_i)$, a contradiction. Thus $e_j=-e_i$ and consequently $5e_i=0$, by \eqref{e5}. So, in this case $5$ is a divisor of $n$. Hence, if $\varphi$ is decomposed into $k$ cycles of the length $2$, then $(Q,+)$ has the order $n=2k+1=5t$. Since $t$ must be odd, we see that in this case $n=10s+5$.

If $\varphi$ contains a cycle of the length $5$, then for some $e_i\in Q$ we have $\varphi^5(e_i)=e_i$ and $\varphi(e_i)\ne -e_i$. This, by Corollary \ref{C-abel}, implies $2e_i=0$. Thus, $2$ is a divisor of $n$ and $n>5$. Moreover, each element of this cycle has order $2$. Therefore, in the case when $\varphi$ is decomposed into disjoint cycles of the length $5$, the group $(Q,+)$ has $5s+1$ elements and all non-zero elements have order $2$. So, $(Q,+)$ is the direct product of copies of $\mathbb{Z}_2$. 
Thus, $n=2^k=5t+1$. So, $t$ is odd and, as in the previous case, $n=10s+6$. 
If $\varphi$ is decomposed into cycles of the length $10$, then obviously $n=10s+1$.

Now, if $\varphi$ is decomposed into cycles of the length $2$ and $5$, then $10$ divides $n$. Thus $n=10s$. If $\varphi$ is decomposed into $p>0$ cycles of the length $2$ and $q>0$ cycles of the length $10$, then $n=2p+10q+1$ and $5$ divides $n$. Hence $n=10s+5$. If $\varphi$ is decomposed into $p>0$ cycles of the length $5$ and $q>0$ cycles of the length $10$, then $n=5p+10q+1$ and $2$ divides $n$. Hence $n=10s+6$. Finally, if $\varphi$ is decomposed into $p>0$ cycles of the length $2$, $q>0$ cycles of the length $5$ and $r$ cycles of the length $10$, then $n=2p+5q+10r+1$ and $5$ divides $n$. Hence $n=10s+5$. 
\end{proof}

\begin{corollary}\label{Z5}
The smallest pentagonal quasigroup is induced by the group $\mathbb{Z}_5$ and has the form 
 $x\cdot y=[4x+2y]_5$.
\end{corollary}
\begin{proof}
Indeed, by Theorem \ref{T-fin}, $\mathbb{Z}_5$ is the smallest group that can be used in the construction of a pentagonal quasigroup. In this group only $a=4$ satisfies \eqref{e7}. Thus, the multiplication of this quasigroup is defined by $x\cdot y=[4x+2y]_5$.
\end{proof}

\begin{proposition}\label{Zn-com}
A groupoid $(Q,\cdot)$ is a commutative pentagonal quasigroup if and only if there exists an abelian group $(Q,+)$ of exponent $11$ such that $x\cdot y=6x+6y$ for all $x,y\in Q$.
\end{proposition}
\begin{proof}  
By Theorem \ref{main} for a commutative pentagonal quasigroup there exists an abelian group $(Q,+)$ and its automorphism $\varphi$ such that $\varphi=\varepsilon-\varphi$. Thus $\varepsilon=2\varphi$. This, by \eqref{e5}, gives $\varphi(\varphi^3-\varphi^2+\varphi+\varepsilon)=0$. Therefore, $\varphi(\varphi^2-\varphi+3\varepsilon)=0$, and consequently, $\varphi^2+5\varphi=0$,  so $\varphi+5\varepsilon=0$. Hence $11\varphi=0$. Thus $11x=0$ for each $x\in Q$. Moreover, from $11\varphi=0$ we obtain $\varphi(x)=-10\varphi(x)=-5x=6x$ and $(\varepsilon-\varphi)(x)=x-6x=-5x=6x$. So, exp$(Q,+)=11$ and $x\cdot y=6x+6y$ for all $x,y\in Q$.

The converse statement is obvious.
\end{proof}
\begin{corollary}\label{Cc} The variety of commutative, idempotent, medial groupoids satisfying the pentagonal identity is the variety of commutative, pentagonal quasigroups, whose spectrum is $\{11^n:n=0,1,2,\ldots\}$. 
\end{corollary}
\begin{proof} It follows from Proposition \ref{Zn-com} that the spectrum of the variety of commutative pentagonal quasigroups is $\{11^n:n=0,1,2,\ldots\}$. So, we need only prove that a commutative, idempotent, medial groupoid $(Q,\cdot)$ satisfying the pentagonal identity is a quasigroup. Let $a,b\in Q$. Then the pentagonal identity ensures that the equations $xa=b$ and $ax=b$ have a solution $x=(ab\cdot a)b$. Suppose that $za=b$. Then $z=(az\cdot a)z\cdot a=(ba\cdot z)a=(ab\cdot z)\cdot aa=(ab\cdot a)b=x$ and the solution is unique.
\end{proof}


\section{Translatable pentagonal quasigroups}\setcounter{section}{4}\setcounter{theorem}{0}

Recall a quasigroup $(Q,\cdot)$, with $Q=\{1,2,\ldots,n\}$ and $1\leqslant k<n$, is {\it $k$-translatable} if its multiplication table is obtained by the following rule: If the first row of the multiplication table is $a_1,a_2,\ldots,a_n$, then the $q$-th row is obtained from the $(q-1)$-st row by taking the last $k$ entries in the $(q-1)-$st row and inserting them as the first $k$ entries of the $q$-th row and by taking the first $n-k$ entries of the $(q-1)$-st row and inserting them as the last $n-k$ entries of the $q$-th row, where $q\in\{2,3,\ldots,n\}$. The multiplication in a $k$-translatable quasigroup is given by the formula $i\cdot j=[i+1]_n\cdot [j+k]_n=a_{{k-ki+j}_n}$ (cf. \cite{DM1,DM2} or \cite{DM3}). Moreover, Lemma 9.1 in \cite{DM1} shows that a quasigroup of the form $x\cdot y=[ax+by]_n$ is $k$ translatable only for $k$ such that $[a+kb]_n=0$. Thus, a pentagonal quasigroup induced by $\mathbb{Z}_n$ can be $k$-translatable only for $k\in\{2,3,\ldots,n-2\}$. 

\begin{theorem}\label{T-trans}
Every pentagonal quasigroup induced by $\mathbb{Z}_n$ is $k$-translatable for some $k>1$ such that $(k,n)=1$. If it has the form $x\cdot y=[ax+(1-a)y]_n$, then is $k$-translatable for $k=[1-a^3-a]_n$.
\end{theorem}
\begin{proof}
Indeed, by Theorem \ref{T-Zn}, $x\cdot y=[ax+(1-a)y]_n$ and $[a^4-a^3+a^2-a+1]_n=0$. Thus, $[a+(-a^3-a+1)(1-a)]_n=0$, which, by Lemma 9.1 from \cite{DM1}, means that this quasigroup is $k$-translatable. Since $k+nt=1-a(a^2+1)=-a^2(a^2+1)$ and $(a,n)=1$, each prime divisor of $k$ and $n$ is a divisor of $a$, which is impossible. So, $(k,n)=1$.
\end{proof}

\begin{theorem}\label{T-trans2} A groupoid $(Q,\cdot)$ of order $n$ is a $k$-translatable pentagonal quasigroup, $k>1$, if and only if it is of the form 
$x\cdot y=[ax+(1-a)y]_n$, where 
\begin{eqnarray}\label{w0}
n|m=k^4-2k^3+4k^2-3k+1 \ \ \ {\rm and} \ \ \ a=[-k^3+k^2-3k+1]_n. 
\end{eqnarray}
\end{theorem}
\begin{proof} Suppose that $(Q,\cdot)$ is a $k$-translatable pentagonal quasigroup of order $n$. By Theorem 4.2 of \cite{DM3} and Lemma 9.1 of \cite{DM1} it is of the form $x\cdot y=[ax+(1-a)y]_n$ and $[a+(1-a)k]_n=0$, where $1<a<n$, $(a,n)=(a-1,n)=1$.  Thus
\begin{eqnarray}\label{w1}
[a+k]_n=[ka]_n \ \ \ \ {\rm and} \ \ \ \ k=[(k-1)a]_n.
\end{eqnarray}
By Theorem \ref{T-trans}, $k=[1-a^3-a]_n$. Therefore, $[a^3]_n=[1-a-k]_n$. So,
\begin{eqnarray}\label{w2}
[ka^3]_n=[k-ka-k^2]_n=[-a-k^2]_n.
\end{eqnarray}
By \eqref{e8}, we also have $[a^5]_n=[-1]_n$. Thus,
\begin{equation*}
[ka^3]_n\stackrel{\eqref{w1}}{=}[(k-1)a^4]_n, \ \ [ka^4]_n=[1-k]_n, \ \ [ka]_n=[(k-1)a^2]_n, \ \ [ka^2]_n=[(k-1)a^3]_n.
\end{equation*}
Therefore, using pentagonality and the above identities, we obtain
$$\arraycolsep=.5mm
\begin{array}{rlll}
0&=[(k-1)(a^4-a^3+a^2-a+1)]_n\\[3pt]
&=[(k-1)a^4-(k-1)a^3+(k-1)a^2-(k-1)a+(k-1)]_n\\
&=[ka^3-a^2-1]_n\stackrel{\eqref{w1}}{=}[-a-k^2-a^2-1]_n.
\end{array}
$$
Hence $[a^2]_n=[-a-k^2-1]_n$, and in the consequence
$$
[a+k]_n=[ka]_n=[(k-1)a^2]_n=[(k-1)(-a-k^2-1)]_n=[-k^3+k^2-2k+1]_n,
$$
which implies the second equation of \eqref{w0}.

The first equation follows from the fact that
$$
0=[a+k-ka]_n=[k^4-2k^3-4k^2-3k+1]_n.
$$

Conversely, let $(Q,\cdot)$ be a groupoid of order $n>1$ with $x\cdot y=[ax+(1-a)y]_n$, where $n$ and $a$ are as in \eqref{w0}.
Then $(a,n)=(a-1,n)=1$. Indeed, each a prime divisor $p$ of $a$ and $n$ is a divisor of $m-a=k^2(k^2-k+3)$. If $p|k$, then, by \eqref{w0}, $p|1$, a contradiction. So, $p|(k^2-k+3)$ and $p/\!\! |\,n$, but then $p|k(k^2-k+3)=1-a$. This also is impossible. Hence $(a,n)=1$. Similarly $(a-1,n)=1$. Thus $1<a<n$ and, in the consequence, $(Q,\cdot)$ is a quasigroup. Since $[a+k(1-a)]_n=0$, by Lemma 9.1 from \cite{DM1}, it is $k$-translatable. This implies \eqref{w1}.

Now, using \eqref{w1} and \eqref{w0}, we obtain
\begin{eqnarray}\label{w3}
[k^2a]_n=[k(ka)]_n=[k(k+a)]_n=[k^2+k+a]_n\stackrel{\eqref{w0}}{=}[-k^3+2k^2-2k+1]_n
\end{eqnarray}

\noindent
and

$\arraycolsep=.5mm
\begin{array}{rlrr}
[k^2a^2]_n=&[(k^2a)a]_n\stackrel{\eqref{w3}}{=}[-k^2(k+a)+2k(k+a)-2ka+a]_n\\[2pt]
=&[-k^3-k^2a+2k^2+a]_n\stackrel{\eqref{w3}}{=}[-k^3-k^2-k-a+2k^2+a]_n\\[4pt]
=&[-k^3+k^2-k]_n .
\end{array}
$

\medskip\noindent
That is,
\begin{eqnarray*}\label{w4}
[k^2a^2]_n=[-k^3+k^2-k]_n \ \ \ {\rm and } \ \ \ [k^3a^2]_n=[-k^4+k^3-k^2]_n .
\end{eqnarray*}
Then \\

\arraycolsep=.5mm
$\begin{array}{rllll}
[a^2]_n&\stackrel{\eqref{w0}}{=}[-kk^2a+k^2a-3ka+a]_n\\
&\stackrel{\eqref{w2},\eqref{w1}}{=}\!\![(k^4-2k^3+2k^2-k)+(-k^3+2k^2-2k+1)+(-3k-3a+a)]_n\\
&\stackrel{\eqref{w0}}{=}[-k^3-3k-2a]_n\stackrel{\eqref{w0}}{=}[k^3-2k^2+3k-2]_n.
\end{array}
$

\medskip\noindent
Consequently,
\begin{eqnarray*}\label{w5}
&&[ka^2]_n= [k^4-2k^3+3k^2-2k]_n\stackrel{\eqref{w0}}{=}[-k^2+k-1]_n.
\end{eqnarray*}
Now, using the above identities, we obtain
\begin{eqnarray*}
&&[a^3]_n= [-k^3a^2+k^2a^2-3ka^2+a^2]_n=[k^3-k^2+2k]_n\stackrel{\eqref{w0}}{=}[1-a-k]_n.
\end{eqnarray*}
Therefore, $[a^4]_n=[a-a^2-ak]_n$  and 
$$
[a^4-a^3+a^2-a+1]_n=[1-ak-a^3]_n=[a+k(1-a)]_n=0,
$$
which, by Theorem \ref{T-Zn}, shows that $(Q,\cdot)$ is a pentagonal quasigroup. 
\end{proof}

\begin{corollary}
For every $k>1$ there exist at least one $k$-translatable pentagonal quasigroup.
\end{corollary}
\begin{proof}
One $k$-translatable pentagonal quasigroup is defined by Theorem \ref{T-trans2}. In this quasigroup $a$ and $n$ are as in \eqref{w0}. If $m$ is a divisor of $n$ and $a=[b]_n$, then $a=[b]_m$. Thus $(\mathbb{Z}_m,\cdot)$ with $x\cdot y=[a'x+(1-a')y]_m$, $a'=[-k^3+k^2-3k+1]_m$, also is a $k$-translatable pentagonal quasigroup.
\end{proof}

According to Theorem \ref{T-trans2} for $k=2$, we have $n=m=11$ and $a=2$. So for $k=2$ there is only one $k$-translatable pentagonal quasigroup induced by $\mathbb{Z}_n$. It has the form $x\cdot y=[2x+10y]_{11}$. For $k=3$, $m=55$, $a=[29]_n$ and $n|m$ there are three $k$-translatable pentagonal quasigroups induced by $\mathbb{Z}_n$. They have the form: $x\cdot y=[29x+27y]_{55}$, $x\cdot y=[7x+5y]_{11}$ and $x\cdot y=[4x+2y]_5$. Other calculations for $k\leqslant 20$ are presented below.

$${\small
\begin{array}{|c|l|l|c|c|}\hline
k&x\cdot y\\ \hline
2&[2x+10y]_{11}\\ \hline 
3&[4x+2y]_5, \;[7x+5y]_{11},\;[29x+27y]_{55}\\ \hline
4&[122x+60y]_{181}\\ \hline
5&[347x+115]_{461}\\ \hline
6&[794x+198y]_{991}\\ \hline
7&[27x+5y]_{31},\;[52x+10y]_{61},\;[1577x+315]_{1891}\\ \hline
8&[190x+472y]_{661},\;[2834x+472y]_{3305}\\ \hline
9&[8x+4y]_{11},\;[308x+184y]_{491},\;[4727x+675y]_{5401}\\ \hline
10&[6x+6y]_{11},\;[593x+169y]_{761},\;[7442x+930y]_{8371}\\ \hline
11&[29x+3y]_{31},\;[362x+40]_{401},\;[11189x+1243y]_{12431}\\ \hline
12&[14x+58y]_{71},\;[138x+114y]_{251},\;[16202x+1620y]_{17821}\\ \hline
13&[25x+17y]_{41},\;[24x+32y]_{55},\;[112x+10y]_{121},\;[189x+17y]_{205},\\
&[354x+252y]_{605},\;[189x+2067y]_{2255},\;[2895x+2067y]_{4961},\;[22739x+2076]_{24805}\\ \hline
14&[472x+2590y]_{3061},\;[31082x+2590y]_{33671}\\ \hline
15&[4x+38y]_{41},\;[79x+1013y]_{1091},\;[41537x+3195y]_{44731}\\ \hline
16&[54434x+3888y]_{58321}\\ \hline
17&[42x+90y]_{131},\;[465x+107y]_{571},\;[70127x+4675y]_{74801}\\ \hline
18&[13350x+5562y]_{18911},\;[88994x+5562y]_{94555}\\ \hline
19&[111437x+6555y]_{117991}\\ \hline
20&[17x+85y]_{101},\;[70x+62y]_{131},\;[118x+994y]_{1111},\;[987x+455y]_{1441},\\ 
&[5572x+7660y]_{13231},\;[137882x+7660y]_{145541}\\ \hline
\end{array}}
$$

Let $\mathbb{Z}_{11}^*$ be the pentagonal quasigroup with the multiplication $x\cdot y=[6x+6y]_{11}$.
By the above result, a finite commutative pentagonal quasigroup is the direct product of $m$ copies of $\mathbb{Z}_{11}^*$ but for $m>1$, as it is shown below, they are not translatable.

\begin{theorem}
$\mathbb{Z}_{11}^*$ is the only translatable commutative pentagonal quasigroup.
\end{theorem}
\begin{proof}
Let $(Q,\cdot)$ be a commutative pentagonal quasigroup. By definition, an infinite quasigroup cannot be translatable. So, $(Q,\cdot)$ must be finite. By Proposition \ref{Zn-com} its order is $n=11^m$.

If $m=1$, then, by Proposition \ref{Zn-com}, the multiplication of $(Q,\cdot)$ has the form $x\cdot y=[6x+6y]_{11}$. From the multiplication table of this quasigroup it follows that it is $k$-translatable for $k=10=n-1$. So, for $m=1$, our theorem is valid.

Now let $m>1$ and $(Q,\cdot)$ be $(n-1)$-translatable. According to Lemma 2.7 in \cite{DM2}, we can assume that $Q$ is ordered in the following way: 
$x^{(1)},x^{(2)},x^{(3)},\ldots,x^{(n)}$, where $x^{(1)}=(1,0,0,\ldots,0)$. Then the multiplication table of $(Q,\cdot)$ has the form 

$$\begin{array}{|c|c|c|c|c|c|c|c||c} \hline
\rule{0pt}{12pt}\cdot&x^{(1)}&x^{(2)}&x^{(3)}&\ldots\ldots&x^{(n)}\\ \hline
\rule{0pt}{12pt}x^{(1)}&x^{(1)}\cdot x^{(1)}&x^{(1)}\cdot x^{(2)}&x^{(1)}\cdot x^{(3)}&\ldots\ldots&x^{(1)}\cdot x^{(n)}\\ \hline
\rule{0pt}{12pt}x^{(2)}&x^{(2)}\cdot x^{(1)}&x^{(2)}\cdot x^{(2)}&x^{(2)}\cdot x^{(3)}&\ldots\ldots&x^{(2)}\cdot x^{(n)}\\ \hline
\rule{0pt}{12pt}x^{(3)}&x^{(3)}\cdot x^{(1)}&x^{(3)}\cdot x^{(2)}&x^{(3)}\cdot x^{(3)}&\ldots\ldots&x^{(3)}\cdot x^{(n)}\\ \hline
\ldots\ldots&\ldots\ldots&\ldots\ldots&\ldots\ldots&\ldots\ldots&\ldots\ldots\\[3pt] \hline
\ldots\ldots&\ldots\ldots&\ldots\ldots&\ldots\ldots&\ldots\ldots&\ldots\ldots\\[3pt] \hline
\end{array}
$$ 
Since $(Q,\cdot)$ is $(n-1)$-translatable, $x^{(2)}\cdot x^{(t)}=x^{(1)}\cdot x^{(t+1)}$ for all $t\in\mathbb{Z}_n$.

Let $x^{(2)}=(a_1,a_2,\ldots,a_m)$. We will prove by induction that 
$$
x^{(t+1)}=(ta_1-(t-1),ta_2,ta_3,\ldots,ta_m) \ \ \ \forall  t\in\mathbb{Z}_n.
$$ 
The induction hypothesis is clearly true, by definition, for $t=1$. Assume that the induction hypothesis is true for all $s\leqslant t$. Then
$$
x^{(t)}=((t-1)a_1-(t-2),(t-1)a_2,(t-1)a_3,\ldots,(t-1)a_m).
$$

Suppose that $x^{(t+1)}=(z_1,z_2,z_3,\ldots,z_m)$. Since $x^{(2)}\cdot x^{(t)}=x^{(1)}\cdot x^{(t+1)}$,
we have $6x^{(2)}+6x^{(t)}=6x^{(1)}+6x^{(t+1)}$. The last expression means that
$$
(6a_1+6(t-1)a_1-6(t-2),6ta_2,6ta_3,\ldots,6ta_m)=(6+6z_1^{(t)},6z_2^{(t)},6z_3^{(t)},\ldots,6z_m^{(t)}).
$$
Hence $6z_1^{(t)}=6\big(ta_1-(t-1)\big),$ which implies $z_1^{(t)}=ta_1-(t-1).$ Also $z_s^{(t)}=ta_s$ for all $s=2,3,\ldots,m$. So, $x^{(t+1)}=(ta_1-(t-1),ta_2,ta_3,\ldots,ta_m)$, as required.

Now, $x^{(12)}=(11a_1-10,11a_2,11a_3,\ldots,11a_m)=(-10,0,0,\ldots,0)=x^{(1)}$, a contradiction because all $x^{(1)},x^{(2)},x^{(3)},\ldots,x^{(n)}$ are different. So, for $m>1$ a quasigroup $(Q,\cdot)$ cannot be $(n-1)$-translatable.
\end{proof}

Suppose that $(G,\cdot)$ is a commutative pentagonal quasigroup and $a,b$ are two distinct elements of $G$. Then it is straighforward to prove that $a$ and $b$ generate the subquasigroup 
$$
\langle a,b\rangle =\{a,\,b,\,ab,\,aba,\,bab,\,aba\!\cdot\!a,\,aba\!\cdot\!b,\, bab\!\cdot\!a,\, bab\!\cdot\!b,\, (aba\!\cdot\!a)b,\,(bab\!\cdot\!b)a\}
$$
and that $\langle a,b\rangle$ is isomorphic to $\mathbb{Z}_{11}^*$. Then we take $c\not\in\langle a,b\rangle$, if $c$ exists.

\begin{lemma} $\langle a,b\rangle\cap\langle b,c\rangle=\{b\}$.
\end{lemma}
\begin{proof} From the multiplication table of $\mathbb{Z}_{11}^*$ we see that any two distinct elements generate $\mathbb{Z}_{11}^*$. Hence, $\langle a,b\rangle\cap\langle b,c\rangle$  cannot contain $b$ and another element of $\langle a,b\rangle\cap\langle b,c\rangle$, or else $c\in\langle a,b\rangle=\langle b,c\rangle$, a contradiction.   
\end{proof}
\begin{theorem} $H=\langle a,b\rangle\langle b,c\rangle$ is a commutative pentagonal subquasigroup of $(G,\cdot)$ isomorphic to $\,\mathbb{Z}_{11}^*\times\mathbb{Z}_{11}^*$.
\end{theorem}
\begin{proof} Since $(G,\cdot)$ is medial, $(\langle a,b\rangle\langle b,c\rangle)(\langle a,b\rangle\langle b,c\rangle)\subseteq\langle a,b\rangle\langle b,c\rangle$. Note that $\langle a,b\rangle\subseteq\langle a,b\rangle b\subseteq\langle a,b\rangle\langle b,c\rangle$ and $\langle b,c\rangle\subseteq b\langle b,c\rangle\subseteq\langle a,b\rangle\langle b,c\rangle$. Hence, the commutative pentagonal quasigroup $H=\langle a,b\rangle\langle b,c\rangle\supseteq\langle a,b\rangle\cup\{c\}$ has more than $11$ elements and less than or equal to $121$ elements. Therefore, as we have already seen, $H$ has $121$ elements and is isomorphic to $\mathbb{Z}_{11}^*\times\mathbb{Z}_{11}^*$.  This completes the proof. 
	\end{proof}
\begin{corollary} $\mathbb{Z}_{11}^*\times\mathbb{Z}_{11}^*$ is generated by three distinct elements.
\end{corollary}
\begin{corollary} If $xy=zw\in\langle a,b\rangle\langle b,c\rangle$, then $x=z$ and $y=w$.
\end{corollary}
\begin{corollary} If $d\not\in\langle a,b\rangle\langle b,c\rangle$, then $(\langle a,b\rangle\langle b,c\rangle)\langle b,d\rangle$ is a commutative pentagonal 
quasigroup of order $11^3$ and is isomorphic to $\mathbb{Z}_{11}^*\times\mathbb{Z}_{11}^*\times\mathbb{Z}_{11}^*$.
\end{corollary}
\begin{corollary} $\mathbb{Z}_{11}^*\times\mathbb{Z}_{11}^*\times\mathbb{Z}_{11}^*$ is generated by four distinct elements.
\end{corollary}
\begin{corollary} The direct product of $n$ copies of $\mathbb{Z}_{11}^*$ is generated by $n+1$ distinct elements.
\end{corollary}

\section{Groups inducing pentagonal quasigroups}\setcounter{section}{5}\setcounter{theorem}{0}

Pentagonal quasigroups are very large. Using Theorem \ref{T-trans2} we can determine all pentagonal quasigroups induced by $\mathbb{Z}_n$. Below we present several such quasi\-groups.
For $a=3$ there is only one such quasigroup. It is induced by the group $\mathbb{Z}_{61}$. Its multiplication is defined by $x\cdot y=[3x-2y]_{61}=[3x+59y]_{61}$. This quasigroup is $32$-translatable. For $a=4$ there are three such quasigroups. They are induced by $\mathbb{Z}_5$, $\mathbb{Z}_{41}$, $\mathbb{Z}_{205}$ and are $3$-, $15$-, $138$-translatable, respectively.

{\small $$
\begin{array}{|c|c|c|c|c|c|c|c|c|c|c|c|c|c|c|c|} \hline
\rule{0pt}{10pt}a&2&3&4&5&6&7&8&9\\ \hline
\rule{0pt}{10pt}n&11&61&5,41,205&521&\!11,101,1111\!&\!11,191,2101&11,331,3641\!&\!1181,5905\!\\ \hline
\rule{0pt}{10pt}k&2&32&3,15,138&392&10,\;82,\,890&3,\;33,\;1752&9,\;143,\;3122&444,5168\\ \hline
\end{array}
$$
$$\begin{array}{|c|c|c|c|c|c|c|c|c|c|c|c|c|c|c|c|} \hline
\rule{0pt}{10pt}a&10&11&12&13&14&15&16\\ \hline
\rule{0pt}{10pt}n&9091&13421&\!19141\!&\!2411\!&\!71,101,355,505,7171,35855\!&\!31,1531,47461\!&\!61681\!\\ \hline
\rule{0pt}{10pt}k&8082&12080&\!17402\!&202&83,\,71,\,83,\,273,\;4414,\,33098&\!21,1204,44072\!&\!57570\!\\ \hline
\end{array}
$$ 
$$\begin{array}{|c|c|c|c|c|c|c|c|c|c|c|c|c|c|c|c|} \hline
\rule{0pt}{10pt}a&17&18&19&20\\ \hline
\rule{0pt}{10pt}n&\!71,101,781,1111,7171,78881\!&\!9041,99451\!&\!55,2251,11255,24761,123805\!&\!152381\!\\ \hline
\rule{0pt}{10pt}k&41,\,19,\,538,\,625,\,2242,\,73952&\!3192,93602\!&\!53,2127,4378,\,17884,116928\!&\!144362\!\\ \hline
\end{array}
$$ 
$$\begin{array}{|c|c|c|c|c|c|c|c|c|c|c|c|c|c|c|c|} \hline
\rule{0pt}{10pt}a&21&22&23&24
\\ \hline
\rule{0pt}{10pt}n&\!185641\!&\!224071\!&\!31,41,211,1271,6541,8651,268181\!&\!55,5791,28955,63701,318505\!
\\ \hline
\rule{0pt}{10pt}k&\!176360\!&\!213402\!&\!25,\,28,\;48,\;\,521,\,893,\,5113,\,255992\!&\!13,2267,15108,49854,304658\!
\\ \hline
\end{array}
$$ }

The above table shows that from groups $\mathbb{Z}_n$ for $n<24$ only groups $\mathbb{Z}_5$ and $\mathbb{Z}_{11}$  determine pentagonal quasigroups. To determine other groups of order $n<100$ inducing pentagonal quasigroups observe that from Corollary \ref{C-abel} and Theorem \ref{T-fin} it follows that an Abelian group inducing a pentagonal quasigroup is the direct product of several copies of the group $\mathbb{Z}_5$ or has a regular automorphism $\varphi\ne\varepsilon$ of order $10$. Observe that from Proposition \ref{cprod}, Corollary \ref{C-ord}, Theorem \ref{T-fin} and the table above the possible values of $n<100$ are $n\in \{5,11,16,25,31,40,41,45,55,56,61,71,80,81\}$.

For $n=5$ we have one pentagonal quasigroup, for $n=11$ there are four such quasigroups (see the above table). For $n=16$ we have five Abelian groups of order $16$: $\mathbb{Z}_{16}$, $\mathbb{Z}_2\times\mathbb{Z}_8$, $\mathbb{Z}_2\times\mathbb{Z}_2\times\mathbb{Z}_{4}$, $\mathbb{Z}_4\times\mathbb{Z}_{4}$ and $\mathbb{Z}_2^4$. From the above table it follows that the group $\mathbb{Z}_{16}$ does not induce any pentagonal quasigroup. Groups $\mathbb{Z}_2\times\mathbb{Z}_8$, $\mathbb{Z}_2\times\mathbb{Z}_2\times\mathbb{Z}_{4}$, $\mathbb{Z}_4\times\mathbb{Z}_{4}$ do not have automorphisms of order $10$ (Theorem \ref{T-Hil}), so they cannot be considered as a group inducing pentagonal quasigroups. The group $\mathbb{Z}_2^4$ can be treated as a vector space $V$ over $\mathbb{Z}_2$. Then, by Corollary \ref{C-abel}, interesting for us automorphisms $\varphi$ are linear endomorphisms of $V$ for which $\lambda=-1$ is an eigenvalue of $\varphi^5$. From these endomorphisms we select those satisfying \eqref{e5}. There is $1\,344$ such endomorphisms, so the group $\mathbb{Z}_2^4$ induces $1\,344$ pentagonal quasigroups.

The group $\mathbb{Z}_{25}$ has four elements of order $5$, namely $5$, $10$, $15$ and $20$. Thus $\varphi(5)\in\{10,15,20\}$. Therefore $\varphi$ restricted to the set $\{5,10,15,20\}$ has the form $\varphi(x)=ax$, where $a\in\{2,3,4\}$, but such $\varphi$ does not satisfy  \eqref{e5}. Hence $\mathbb{Z}_{25}$ does not induce a pentagonal quasigroup. The group $\mathbb{Z}_5\times\mathbb{Z}_5$ induces 24 pentagonal quasigroups. These quasigroups are induced by matrices 
$$
A\in \left\{\left[
\begin{array}{cc}
	0&1\\4&3
\end{array}	\right],
\left[
\begin{array}{cc}
	0&2\\2&3
\end{array}	\right],
\left[
\begin{array}{cc}
	0&3\\3&3
\end{array}	\right],
\left[
\begin{array}{cc}
	0&3\\1&2
\end{array}	\right],
\left[
\begin{array}{cc}
	0&4\\1&3
\end{array}	\right],
\left[
\begin{array}{cc}
	4&0\\1&4
\end{array}\right]
\right\}
$$
and $-A^2, A^3, -A^4$.

Pentagonal quasigroups of order $31$ are induced by the group $\mathbb{Z}_{31}$. They are determined by an automorphism $\varphi(x)=ax$, where $a\in\{15, 23, 27, 29\}$ (see table below).

From Abelian groups of order $40$ the groups $\mathbb{Z}_{40}$, $\mathbb{Z}_2\times\mathbb{Z}_{20}$ and $\mathbb{Z}_4\times\mathbb{Z}_{10}$ have one or three elements of order $2$, so they cannot induce pentagonal quasigroups. In the group $\mathbb{Z}_2\times\mathbb{Z}_2\times\mathbb{Z}_{10}$ only elements $(0,0,2)$, $(0,0,4)$, $(0,0,6)$, $(0,0,8)$ have order $5$. Thus $\varphi(0,0,2)\in\{(0,0,4), (0,0,6), (0,0,8)\}$. But then $\varphi^5(0,0,2)+(0,0,2)\ne (0,0,0)$, a contradiction. Therefore there are no pentagonal quasigroups of order $40$.

Pentagonal quasigroups of order $41$ can be calculated by solution of the equation \eqref{e7} or \eqref{e8}. The solutions are $a=4,23,25,31$. So there are four such quasigroups.

For $n=45$ there are two Abelian groups: $\mathbb{Z}_{45}$ and $\mathbb{Z}_3\times\mathbb{Z}_{15}$. The first group has two elements of order $3$, so by Proposition \ref{ord} it cannot induce pentagonal quasigroups. The second group has four elements of order $5$. The smallest is $(0,3)$. Thus $\varphi(0,3)=(0,3a)$ for $a=2,3,4$. But then $\varphi^5(0,3)+(0,3)\ne (0,0)$. Thus, pentagonal quasigroups of order $45$ do not exist.

For $n=55$ there exists only one Abelian group: $\mathbb{Z}_{55}$. Its automorphisms have form $\varphi(x)=ax$, where $(a,55)=1$. The automorphisms inducing pentagonal quasigrous should satisfy \eqref{e8}. It is easily to see, that for $k=0,1,2,...,9$ the last digit of $k^5$ is $k$. So, for $a=mk$ the last digit of $a^5$ also is $k$. Since $a^5+1$ must be divided by $5$, $a=m4$ or $a=m9$. 
The above table shows that the smallest possible value of $a$ is $19$. Because $44$ is divided by $11$, $44^5+1$ cannot be divided by $11$. Thus $44$ should be omitted. Also $54=(-1)({\rm mod}\,55)$ should be omitted.
By direct calculation we can see that from other $a<54$ acceptable are $24$, $29$ and $39$.
Hence there are four pentagonal quasigroups of order $55$. They are isomorphic to the direct product of pentagonal quasigroups induced by $\mathbb{Z}_5$ and $\mathbb{Z}_{11}$.

From Abelian groups of order $56$ groups $\mathbb{Z}_{56}$, $\mathbb{Z}_2\times\mathbb{Z}_{28}$, $\mathbb{Z}_4\times\mathbb{Z}_{14}$ have one or three elements of order $2$. Thus they cannot induce pentagonal quasigroups. The group $\mathbb{Z}_2\times\mathbb{Z}_2\times\mathbb{Z}_{14}$ has six elements of order $7$. In the same manner an in the case of groups of order $45$ we can prove that this group cannot induce pentagonal quasigroups. 
 
Pentagonal quasigroups of prime orders $61$ and $71$ can be calculated in the same way as for $n=41$. Results are presented in the table below.

An Abelian group $G$ of order $80$ can be decomposed into the direct product of two groups $H$ and $\mathbb{Z}_5$, where $H$ is a group of order $16$. From groups of order $16$ only $\mathbb{Z}_2^4$ induces pentagonal quasigroups. So, by Proposition \ref{cprod}, from groups of order $81$ only $\mathbb{Z}_2^4\times\mathbb{Z}_5$ induces pentaginal quasigroups. We have $1\,344$ such quasigroups.

The group $\mathbb{Z}_{81}$ has only two elements of order $2$, so, by Proposition \ref{ord}, this group cannot be inducing group for a pentagonal quasigroup. Theorem \ref{T-Hil} shows that from other Abelian groups of order $81$ only the group $\mathbb{Z}_3^4$ can have an automorphism of order $10$. Using a computer software we calculate $303\,264$ of such automorphisms satisfying \eqref{e5}. So, $\mathbb{Z}_3^4$ induces $303\, 264$ pentagonal quasigroups.

\medskip
In this way we have proved the following:
\begin{theorem} The groups of order $n<100$ that induce pentagonal quasigroups are $\mathbb{Z}_{5}$, $\mathbb{Z}_{11}$, $\mathbb{Z}_{2}^4$, $\mathbb{Z}_{5}^2$, $\mathbb{Z}_{31}$, $\mathbb{Z}_{41}$, $\mathbb{Z}_{55}$, $\mathbb{Z}_{61}$, $\mathbb{Z}_{71}$, $\mathbb{Z}_{2}^4\times \mathbb{Z}_{5}$ and $\mathbb{Z}_{3}^4$. 
\end{theorem}

For $n<100$ pentagonal quasigroups induced by $\mathbb{Z}_n$ are as follows:
 
$${\small
\begin{array}{|c|c|c|c|c|c|c|c||c} \hline
\rule{0pt}{10pt}n=5&4x+2y&&&\\ \hline
\rule{0pt}{10pt}n=11&2x+10y&6x+6y&7x+5y&8x+4y\\ \hline
\rule{0pt}{10pt}n=31&15x+17y&23x+9y&27x+5y&29x+3y\\ \hline
\rule{0pt}{10pt}n=41&4x+38y&23x+19y&25x+17y&31x+11y\\ \hline
\rule{0pt}{10pt}n=55&19x+37y&24x+32y&29x+27y&39x+17y\\ \hline
\rule{0pt}{10pt}n=61&3x+59y&27x+35y&41x+21y&52x+10y\\[3pt] \hline
\rule{0pt}{10pt}\;\;n=71\;\;&14x+58y&17x+55y&46x+26y&66x+6y\\[3pt] \hline
\end{array}}
$$


\section{Parastrophes of pentagonal quasigroups}\setcounter{section}{6}\setcounter{theorem}{0}

Each quasigroup $(Q,\cdot)$ determines five new quasigroups $Q_i=(Q,\circ_i)$ with the operations $\circ_i$ defined as follows:
$$
\begin{array}{cccc}
x\circ_1 y=z\longleftrightarrow x\cdot z=y\\
x\circ_2 y=z\longleftrightarrow z\cdot y=x\\
x\circ_3 y=z\longleftrightarrow z\cdot x=y\\
x\circ_4 y=z\longleftrightarrow y\cdot z=x\\
x\circ_5 y=z\longleftrightarrow y\cdot x=z\\
\end{array}
$$
Such defined (not necessarily distinct) quasigroups are called {\em parastrophes} or {\em conjugates} of $Q$.

Parastrophes of each quasigroup can be divided into separate classes containing isotopic parastrophes. The number of such classes is always $1$, $2$, $3$ or $6$ (cf. \cite{D'15}). In some cases (described in \cite{Lind}) parastrophes of a given quasigroup $Q$ are pairwise equal. 
Parastrophes do not save properties of the initial quasigroup. Parastrophes of an idempotent quasigroup are idempotent quasigroups, but parastrophes of a pentagonal quasigroup are not pentagonal quasigroups, in general.

Let $(Q,\cdot)$ be a pentagonal quasigroup induced by the group $\mathbb{Z}_n$. Then $x\cdot y =[ax+(1-a)y]_n$ and $[a^4-a^3+a^2-a+1]_n=0$.
Such quasigroup is $k$-translatable for $k=[1-a^3-a]_n$. Since $[a(1-a+a^2-a^3)]_n=1=[(1-a)(a^3+a)]_n$, from Theorems 5.1 and 5.3 in \cite{DM3} we obtain the following characterization of parastrophes of pentagonal quasigroups.

\begin{proposition}\label{ptrans}
If $(Q,\cdot)$ is a pentagonal quasigroup with multiplication $x\cdot y=[ax+(1-a)y]_n$, then its parastrophe

\medskip
$\arraycolsep=.5mm
\begin{array}{rlllcccc}
x\circ_1 y=&[(1-a^3-a)x+(a^3+a)y]_n \ &{\rm is } \ k{\rm -translatable \ for} \ k=a,\\[3pt]
x\circ_2 y=&[-a^4x+(a^4+1)y]_n &{\rm is } \ k{\rm -translatable\ for} \ k=[a^3+a]_n\,,\\[3pt]
x\circ_3 y=&[(a^4+1)x+(-a^4)y]_n &{\rm is } \ k{\rm -translatable\ for} \ k=[1-a]_n\,,\\[3pt]
x\circ_4 y=&[(a^3+a)x+(1-a-a^3))y]_n \;&{\rm is } \ k{\rm -translatable\ for} \ k=[-a^4]_n\,,\\[3pt]
x\circ_5 y=&[(1-a)x+a)y]_n &{\rm is } \ k{\rm -translatable\ for} \ k=[a^4+1]_n\,.
\end{array}
$
\end{proposition}

Using Proposition \ref{ptrans} we can show for which values of $a$ and $n$ parastrophes of a pentagonal quasigroup with the multiplication $x\cdot y=[ax+(1-a)y]_n$ are pentagonal, quadratical ($xy\cdot x=zx\cdot yz$), hexagonal ($x\cdot yx=y$),  GS-quasigroups ($x(xy\cdot z)\cdot z=y$), ARO-quasigroups ($xy\cdot y=yx\cdot x$), Stein quasigroups ($x\cdot xy=yx$), right modular ($xy\cdot z=zy\cdot x$) and C3 quasigroups ($(xy\cdot y)y=x$).

We start with the lemma that is a consequence of our results proved in \cite{DM3}.

\begin{lemma}\label{Lpara}
Let $(Q,\cdot)$ be a quasigroup of the form $x\cdot y=[ax+(1-a)y]_n$. Then 

\smallskip
$[2a^2-2a+1]_n=0$ \ if it is quadratical $($Theorem $4.8$ in $\cite{DM'16})$, 

\smallskip $[a^2-a+1]_n=0$ \ if it is hexagonal, 

\smallskip $[a^2-a-1]_n=0$ \ if it is a $GS$-quasigroup, 

\smallskip $[2a^2]_n=1$ \ if it is an $ARO$-quasigroup, 

\smallskip $[a^2-3a+1]_n=0$ \ if it is a Stein quasigroup,

\smallskip $[a^2+a-1]_n=0$ \ if it is right modular, 

\smallskip $[a^3]_n=1$ \ if it is a $C3$ quasigroup. 
\end{lemma}

Using the above characterization and the fact that a quasigroup of the form $x\cdot y=[ax+(1-a)y]_n$ is $k$-translatable if and only if $[a+(1-a)k]_n=0$ (cf. \cite{DM1}, \cite{DM2} or \cite{DM3}) we obtain 
\begin{lemma}\label{Ltrans}
A quasigroup of the form $x\cdot y=[ax+(1-a)y]_n$ is  

\smallskip $[1-2a]_n$-translatable if and only if it is quadratical,

\smallskip $[1-a]_n$-transltable if and only if it is hexagonal,

\smallskip $[a+1]_n$-translatable if and only if it is a $GS$-quasigroup,  

\smallskip $[-2a-1]_n$-translatable if and only if if it is an $ARO$-quasigroup, 

\smallskip $[a-1]_n$-translatable if and only if it is a Stein quasigroup,

\smallskip $[-1-a]_n$-translatable if and only if it is right modular.

\smallskip\noindent 
A $C3$ quasigroup is $k$-translatable for $k$ such that $[(1-a^2)k]_n=1$. 
\end{lemma}

Using these two lemmas we can determine properties of parastrophes of pentagonal quasigroups induced by $\mathbb{Z}_n$.
We start with $Q_1$.

$\bullet$ Suppose that $Q_1$ is pentagonal. Then $a = [1-(1-a^3-a)^3-(1-a^3-a)]_n$, from translatability, and $[(1-a^3-a)^2]_n=[-a^2-a-1]_n$, from \eqref{e7}. Then we have $[(1-a^3-a)^3]_n=[(-a^2-a-1) (1-a^3-a)]_n=[a^4+2a^3-2]_n\stackrel{\eqref{e7}}{=}[3a^3-a^2+a-3]_n.$
Therefore, $a=[1-(1-a^3-a)^3-(1-a^3-a)]_n=[-2a^3+a^2+3]_n$, whence, multiplying by $a^2$, we obtain $[a^4-a^3]_n=[-3a^2-2]_n$. This, by \eqref{e7}, shows that $[2a^2+a+1]_n=0$. Multiplying this equation by $a^3$ and applying \eqref{e7} we get $[a^4+a^3]_n=2$. Adding this equation to $[a^4-a^3]_n=[-3a^2-2]_n$ we obtain $[2a^4]_n=[-3a^2]_n$. Thus, $[2a^2]_n=[-3]_n$ and consequently, $[a^4-a^3]_n=[-a^2-2a^2-2]_n=[-a^2+1]_n$. Hence $[a^4-a^3+a^2]_n=1$, which by \eqref{e7} implies $a=2$ and $n=11$.

 $\bullet$ Suppose that $Q_1$ is quadratical. Then, $a=[1-2(1-a^3-a)]_n$ by Lemmas \ref{Lpara} and \ref{Ltrans}. Hence $[2a^3]_n=[1-a]_n$. Also $0=[2(1-a^3-a)^2-2(1-a^3-a)+1]_n=[2a^4-2a-1]_n$. So, $[2a+1]_n=[2a^3a]_n=[(1-a)a]_n=[a-a^2]_n$. Consequently, $[a^2]_n=[-a-1]_n$ and $0=[2a^4-2a^3+2a^2-2a+2]_n=[(2a+1)-(1-a)+2(-a-1)-2a+2]_n=[-a]_n$, contradiction. So, $Q_1$ cannot be quadratical.

$\bullet$ $Q_1$ is never hexagonal. Indeed, $Q_1$ is $a$-translatable and $[a^3+a]_n$-translatable as a hexagonal quasigroup. Hence, $a=[a^3+a]_n$, which implies $[a^3]_n=0$. Thus $0=[a^5]_n=[-1]_n$, a contradiction.

$\bullet$ If $Q_1$ is a $GS$-quasigroup, then $a=[(1-a^3-a)+1]_n$. Hence $[a^3]_n=[2-2a]_n$, $[a^4]_n=[2a-2a^2]_n$, $[-1]_n=[a^5]_n=[2a^2-2a^3]_n=[2a^2+4a-4]_n$, $[2a^2]_n=[3-4a]_n$, $[a^4]_n=[6a-3]_n$. Then $0=[2a^4-2a^3+2a^2-2a+2]_n=[10a-5]_n$. So, $[5a]_n=[10a^2]_n=[15-20a]_n$,
i.e., $[25a]_n=[15]_n$. Thus, $[5a]_n=5$ and $5=[10a]_n=[5a+5a]_n=[10]_n$. Therefore $n=5$ and $x\cdot y=[4x+2y]_5$.

$\bullet$ If $Q_1$ is a $ARO$-quasigroup, then $[2(1-a^3-a)^2]_n=1$, so $[2a^2]_n=[-2a-3]_n$. Also $a=[-2(1-a^3-a)-1]_n$. Thus, $[2a^3]_n=[3-a]_n$, $[2a^4]_n=[3a-a^2]_n$ and $0=[2a^4-2a^3+2a^2-2a+2]_n=[-a^2-4]_n$, i.e., $[a^2]_n=[-4]_n$. Hence $[-8]_n=[2a^2]_n=[-2a-3]_n$ which gives $[2a]_n=5$. So, $[-16]_n=[4a^2]_n=[25]_n$.  Therefore, $n=41$, $a=23$.

$\bullet$ If $Q_1$ is a Stein quasigroup, then $a=[(1-a^3-a)-]_n$. So, $[a^3]_n=[-2a]_n$, $[a^3]_n=[-2a]_n$, $[a^4]_n=[-2a^2]_n$, $[-1]_n=[a^5]_n=[-2a^3]_n$, $a=[2a^4]_n$, $[a^2]_n=[-2]_n$, $a=[2a^4]_n=[8]_n$, $a^3]_n=[-4a^4]_n=[-16]_n$
. Thus, by \eqref{e7}, we obtain $[11]_n=0$. Hence $n=11$ and $a=8$.

$\bullet$ If $Q_1$ is right modular,  then $a=[-1-(1-a^3-a)]_n$. Hence $[a^3]_n=2$, $[a^4]_n=[2a]_n$, $[-1]_n=[a^5]_n=[2a^2]_n$, $a=[-2a^3]_n=[-4]_n$. This by \eqref{e7} implies $n=11$, $a=7$.

$\bullet$ If $Q_1$ is a $C3$ quasigroup, then $1=[(1-(1-a^3-a)^2)a]_n$. Hence $[a^3+a^2+2a-1]_n=0$, which, by \eqref{e7}, gives $[a^4+2a^2+a]_n=0$. So, $[-1+2a^3+a^2]_n=0$. Comparing this equation with $[a^3+a^2+2a-1]_n=0$ we obtain $[a^3]_n=[2a]_n$. So, $[-1]_n=[2a^3]_n=[4a]_n$ 
and $1=[2a^3+a^2]_n=[-1+a^2]_n$. Thus $[a^2]_n=2$ and $[-a]_n=[4a^2]_n=8$. Therefore $2=[a^2]_n=[64]_n$. Consequently, $n=31$, $a=23$.

In other cases the proof is very similar, so we omit it.

\medskip
The result of calculations is presented in the table below.
In this table the intersection of the $ARO$-row with the $Q_3$-column means that for a pentagonal quasigroup $Q$ its parastrophe $Q_3$ is an $ARO$-quasigroup only in the case when $x\cdot y=[14x+58y]_{71}$.
{\small
$$
\begin{array}{|c|c|c|c|c|c|c|c|c|c|c|c|}\hline
&Q&Q_1&Q_2&Q_3&Q_4&Q_5\\[2pt] \hline
\rule{0pt}{10pt}\!\!pentaq.\!\!&&\![2x\!+\!10y]_{11}\! &\!always\!&\!never\!&\![6x\!+\!6y]_{11}\!&\![6x\!+\!6y]_{11}\!\\[2pt] \hline
\rule{0pt}{10pt}\!\!quadrat.\!\!&\![4x\!+\!2y]_5\!&\!never\!&\![4x\!+\!2y]_5\!&\![4x\!+\!2y]_5\!&\!never\!&\![4x\!+\!2y]_5\!\\[2pt] \hline 
\rule{0pt}{10pt}\!\!hexag.\!\!&\!never\!&\!never\!&\!never\!&\!never\!&\!never\!&\!never\!\\[1pt] \hline
\rule{0pt}{10pt}GS&\![8x\!+\!4y]_{11}\!&\![4x\!+\!2y]_5\!&\![7x\!+\!5y]_{11}\!&\![7x\!+\!5y]_{11}\!&\![4x\!+\!2y]_5\!&\![8x\!+\!4y]_{11}\!\\[2pt] \hline
\rule{0pt}{10pt}ARO&\!\![27x\!+\!5y]_{31}\!\!&\!\![23x\!+\!19y]_{41}\!\!&\!\![23x\!+\!9y]_{31}\!\!&\!\![14x\!+\!58y]_{71}\!\!&\!\![25x\!+\!17y]_{41}\!\!&\!\![66x\!+\!6y]_{71}\!\!\\[2pt] \hline
\rule{0pt}{10pt}Stein&\![4x\!+\!2y]_5\!&\![8x\!+\!4y]_{11}\!&\!never\!&\![8x\!+\!4y]_{11}\!&\![7x\!+\!5y]_{11}\!&\![7x\!+\!5y]_{11}\!\\[2pt] \hline
\rule{0pt}{10pt}\!\!r.\, mod.\!\!&\![7x\!+\!5y]_{11}\!&\![7x\!+\!5y]_{11}\!&\!never\!&\![4x\!+\!2y]_5\!&\![8x\!+\!4y]_{11}\!&\![4x\!+\!2y]_5\!\\[2pt] \hline
\rule{0pt}{10pt}C3&\!never\!&\![23x\!+\!9y]_{31}\!&\!never\!&\![23x\!+\!9y]_{31}\!&\![27x\!+\!5y]_{31}\!&\!\![27x\!+\!5y]_{31}\!\!\\[2pt]  \hline
\end{array}
$$
}

\noindent
W.A. Dudek \\
 Faculty of Pure and Applied Mathematics,\\
 Wroclaw University of Science and Technology,\\
 50-370 Wroclaw,  Poland \\
 Email: wieslaw.dudek@pwr.edu.pl\\[4pt]
R.A.R. Monzo\\
Flat 10, Albert Mansions, Crouch Hill,\\ London N8 9RE, United Kingdom\\
E-mail: bobmonzo@talktalk.net

\end{document}